\pdfoutput=1
\documentclass[12pt,reqno]{amsart}
\usepackage{latexsym}
\usepackage{amsmath}
\usepackage{amsthm}
\usepackage{amssymb}
\usepackage{graphicx}
\usepackage{amsfonts}
\usepackage{comment}
\usepackage{hyperref}

\newcommand{\R}{\mathbb{R}}
\newcommand{\doo}{\partial}
\newcommand{\der}{\mathrm{d}}
\newcommand{\ol}[1]{\overline{#1}}
\newcommand{\sulut}[1]{\left( #1 \right)}
\newcommand{\aaltosulut}[1]{\left\{ #1 \right\}}

\DeclareMathOperator{\expo}{exp}
\DeclareMathOperator{\co}{co}

\DeclareMathOperator*{\argmax}{arg \, max}
\DeclareMathOperator{\dist}{dist}

\begin{document}
\theoremstyle{plain}
\newtheorem{MainThm}{Theorem}
\newtheorem{thm}{Theorem}[section]
\newtheorem{clry}[thm]{Corollary}
\newtheorem{prop}[thm]{Proposition}
\newtheorem{lemma}[thm]{Lemma}
\newtheorem{deft}[thm]{Definition}
\newtheorem{hyp}{Assumption}
\newtheorem*{conjecture}{Conjecture}
\newtheorem{question}{Question}
\newtheorem{corollary}[thm]{Corollary}

\newtheorem{claim}[thm]{Claim}

\theoremstyle{definition}
\newtheorem*{definition}{Definition}
\newtheorem{assumption}{Assumption}
\newtheorem{rem}[thm]{Remark}
\newtheorem*{remark}{Remark}
\newtheorem*{acknow}{Acknowledgements}

\newtheorem{example}[thm]{Example}
\newtheorem*{examplenonum}{Example}
\numberwithin{equation}{section}
\newcommand{\nocontentsline}[3]{}
\newcommand{\tocless}[2]{\bgroup\let\addcontentsline=\nocontentsline#1{#2}\egroup}
\newcommand{\eps}{\varepsilon}
\renewcommand{\phi}{\varphi}
\renewcommand{\d}{\partial}
\newcommand{\re}{\mathop{\rm Re} }
\newcommand{\im}{\mathop{\rm Im}}
\newcommand{\mR}{\mathbb{R}}
\newcommand{\mC}{\mathbb{C}}
\newcommand{\mN}{\mathbb{N}} 
\newcommand{\mZ}{\mathbb{Z}} 
\newcommand{\mK}{\mathbb{K}}
\newcommand{\supp}{\mathop{\rm supp}}
\newcommand{\abs}[1]{\left| #1 \right|}
\newcommand{\norm}[1]{\lVert #1 \rVert}
\newcommand{\csubset}{\Subset}
\newcommand{\detg}{\lvert g \rvert}
\newcommand{\msetminus}{\setminus}

\newcommand{\br}[1]{\langle #1 \rangle}

\newcommand{\ehat}{\,\hat{\rule{0pt}{6pt}}\,}
\newcommand{\echeck}{\,\check{\rule{0pt}{6pt}}\,}
\newcommand{\etilde}{\,\tilde{\rule{0pt}{6pt}}\,}

\newcommand{\tr}{\mathrm{tr}}
\newcommand{\mdiv}{\mathrm{div}}

\title{Enclosure method for the $p$-Laplace equation}

\author{Tommi Brander}
\address{Department of Mathematics and Statistics, University of Jyv\"askyl\"a}
\email{tommi.o.brander@jyu.fi}

\author{Manas Kar}
\address{Department of Mathematics and Statistics, University of Jyv\"askyl\"a}
\email{manas.m.kar@maths.jyu.fi}

\author{Mikko Salo}
\address{Department of Mathematics and Statistics, University of Jyv\"askyl\"a}
\email{mikko.j.salo@jyu.fi}

\date{\today}

%\maketitle % comment for amsart

\begin{abstract}
We study the enclosure method for the $p$-Calder\'on problem, which is a nonlinear generalization of the inverse conductivity problem due to Calder\'on that involves the $p$-Laplace equation. 
The method allows one to reconstruct the convex hull of an inclusion in the nonlinear model by using exponentially growing solutions introduced by Wolff.
We justify this method for the penetrable obstacle case, where the inclusion is modelled as a jump in the conductivity. The result is based on a monotonicity inequality and the properties of the Wolff solutions.
\end{abstract}

\maketitle % comment when not using amsart

\begin{comment}
\subjclass[2000]{
Primary
35R30, %	inverse problems of PDE
35J92%   	Quasilinear elliptic equations with $p$-Laplacian
Secondary
 	35Q60%,   	PDEs in connection with optics and electromagnetic theory
}
\end{comment}

%\keywords{Calder\'on's problem, inverse problems, p-laplace equation, enclosure method}

\tableofcontents

\section{Introduction}
We develop a reconstruction formula for identifying the shape or location of an obstacle, embedded in a known background medium, from the boundary measurements for an underlying non-linear PDE. There is a large literature concerning the case where the underlying equation is linear, and the related questions have applications to geophysical problems (detection of mines~\cite{Church:Gagnon:McFee:Wort:2006} or minerals~\cite{Parker:1984} within the earth), bio-medical imaging (for example EIT~\cite{Uhlmann:2009} and coupled physics imaging methods~\cite{Arridge:Scherzer:2012}), radar technology etc.

Several methods have been proposed to reconstruct the shape of the obstacle in the linear case. We first  mention the sampling and probing methods which are based on Isakov's idea of using singular solutions of the elliptic PDE to reconstruct the obstacles included in a known background medium~\cite{Isakov:1988}.
Among the sampling methods we refer to the work of Cakoni-Colton~\cite{Cakoni:Colton:2006}, Colton-Kirsch~\cite{Colton:Kirsch:1996} for the linear sampling method, and Kirsch-Grinberg~\cite{Kirsch:Grindberg:2008} and Harrach~\cite{Harrach:2013} for the factorization method.
Related to probing methods we cite the probe method by Ikehata~\cite{Ikehata:1998} and the singular source method by Potthast~\cite{Potthast:2001}.
Due to the use of Green's type or singular solutions of the forward problem, 
an inconvenient step of probing methods is the need to use approximating domains isolating the source point of the use point sources. This step is quite inconvenient since to perform this method one needs to avoid the unknown obstacle, see~\cite{Potthast:2005} and references therein.
To deal with this issue Ikehata proposed the enclosure method~\cite{Ikehata:1999}, which uses complex geometrical optics (CGO) solutions of the forward problem in place of point sources.
Other methods include those based on oscillating-decaying solutions~\cite{Nakamura:Uhlmann:Wang:2005} or monotonicity arguments~\cite{Harrach:Ullrich:2013}. Our main concern in this paper is to study the enclosure method for a nonlinear equation.

Let us describe the enclosure method for the linear problem.
Here we consider the conductivity problem with homogeneous background. Let $\Omega\subset\mathbb{R}^n, n\geq 2$, be a bounded open set with Lipschitz boundary. Assume that the inclusion $D\subset\Omega$ is an open set (not necessarily connected) with Lipschitz boundary.
The conductivity of the obstacle is taken to be jump discontinuous along $\partial D$, i.e., we consider $\sigma(x) := 1 + \sigma_D(x)\chi_D(x)$, a measurable function, where
\begin{equation}
\sigma_D \in L^{\infty}_+(\Omega) = \left\{\gamma \in L^\infty(\Omega); \gamma > \eps > 0 \text{ for some $\eps \in \R$} \right\},
\end{equation}
and $\chi_D$ is the characteristic function of $D$. Therefore, we formulate the following Dirichlet boundary value problem for the penetrable obstacle case:
\begin{equation}  \label{cal}
\begin{cases}
\mdiv(\sigma(x) \nabla u) = 0 \ \text{in}\ \Omega, \\
u = f \ \text{on}\ \partial\Omega.
\end{cases}
\end{equation}
Given boundary data $f \in H^{1/2}(\partial \Omega)$, the above problem~\eqref{cal} is well posed in $H^1(\Omega)$. Hence we define the voltage to current map, known also as the Dirichlet-to-Neumann map (the DN map for short), formally by 
\[
\Lambda_{\sigma} f := \sigma \frac{\partial u}{\partial\nu}|_{\partial\Omega}
\]
where $u$ satisfies the conductivity problem~\eqref{cal}. Using the weak definition (see Section~\ref{basic}), the DN map becomes a bounded map $$\Lambda_{\sigma} \colon H^{1/2}(\partial\Omega) \rightarrow H^{-1/2}(\partial\Omega).$$

The inverse problem in this formulation is to reconstruct the shape and location of an unknown obstacle $D$ from the knowledge of the DN map $\Lambda_{\sigma}$. The enclosure method introduced by Ikehata~\cite{Ikehata:1999} uses CGO solutions with linear phase for the Laplace equation to detect the convex hull of $D$. The principal idea behind this method is to analyze the behavior of an indicator function, defined via the difference of the DN maps $\Lambda_{\sigma} - \Lambda_1$ and the boundary values of CGO solutions, to decide whether or not the level set of the linear phase function touches the surface of the obstacle. Taking all possible level sets touching the interface produces the convex hull of the obstacle. Note that if $D$ consists of several disjoint parts, the method only gives the convex hull of the set $D$.

There is an extensive literature in this direction for the linear models, see for instance~\cite{Ikehata:1998, Ikehata:2000,Ikehata:2010,Wang:Zhou:2013} and the references therein, for an overview. Here we would like to mention a few of the results.
Using CGO solutions with spherical phase functions for the scalar Helmholtz model, a reconstruction scheme has been proposed by Nakamura and Yoshida~\cite{Nakamura:Yoshida:2007} to detect some non-convex parts of the impenetrable obstacle from the DN map in $\mathbb{R}^n, n=2,3$. In two dimensions the scalar problem has been studied by Nagayasu-Uhlmann-Wang~\cite{Nagayasu:Uhlmann:Wang:2011}, where they used CGO solutions with harmonic polynomial phases. In the recent work by Sini and Yoshida~\cite{Sini:Yoshida:2012}, both the penetrable and impenetrable obstacle cases were considered and some earlier curvature conditions on the boundary of the obstacles were removed. Concerning the Maxwell model, the enclosure method has been studied by Zhou~\cite{Zhou:2010} and Kar and Sini~\cite{Kar:Sini:2014}.
We also cite several other works related to the stationary models with fixed frequencies, see for instance~\cite{Ikehata:1999,Ikehata:2002} and references therein.

In analogy with the linear model, our interest in this paper is to consider the weighted $p$-harmonic model.
In the linear model we have a linear Ohm's law (or Fourier's law of heat conduction); the current~$j$ is proportional to the conductivity~$\sigma$ and the gradient of the potential~$u$:
\begin{equation}
j = -\sigma \nabla u.
\end{equation}
In the $p$-harmonic  model the relation between current and potential is not linear; rather, we have
\begin{equation}
j = -\sigma \abs{\nabla u}^{p-2}\nabla u
\end{equation}
with $1 < p< \infty$, where $p$ does not need to be an integer.
If $p=2$, we recover the linear model.
We combine the nonlinear Ohm's law and Kirchhoff's law, which states that the current~$j$ is divergence-free, to reach the weighted $p$-Laplace equation
\[
\mdiv(\sigma \abs{\nabla u}^{p-2} \nabla u) = 0.
\]
If $\sigma \equiv 1$, this is the $p$-Laplace equation and solutions are called $p$-harmonic functions.

The $p$-Laplace equation is used to study nonlinear dielectrics~\cite{Bueno:Longo:Varela:2008,Garroni:Kohn:2003,Garroni:Nesi:Ponsiglione:2001,Kohn:Levy:1998,Talbot:Willis:1994:a,Talbot:Willis:1994:b} and plastic moulding~\cite{Aronsson:1996}. It is also used to model electro-rheological and thermo-rheological fluids~\cite{Antontsev:Rodrigues:2006,Berselli:Diening:Ruzicka:2010,Ruzicka:2000}, fluids governed by a power law~\cite{Aronsson:Janfalk:1992}, viscous flows in glaciology~\cite{Glowinski:Rappaz:2003} and some plasticity phenomena~\cite{Atkinson:Champion:1984,Idiart:2008,PonteCastaneda:Suquet:1998,PonteCastaneda:Willis:1985,Suquet:1993}.
The $n$-Laplacian is used in conformal geometry~\cite{Liimatainen:Salo:2012}.
The formal $0$-Laplacian is used in ultrasound mediated EIT~\cite{Ammari:Bonnetier:Capdebocsg:Fink:Tanter:2008,Bal:2012,Bal:Schotland:2010,Gebauer:Scherzer:2008} and the formal $1$-Laplacian in conductivity density imaging~\cite{Hoell:Moradifan:Nachman:2013,Henkelman:Joy:Scott:1989,Kim:Kwon:Seo:Yoon:2002,Nachman:Tamasan:Timonov:2009,Armstrong:Henkelman:Joy:Scott:1991}.
The limiting case $p = \infty$ is also of mathematical interest \cite{Evans:2007}.
In the present article, we consider the case where $1 < p < \infty$.

Given a bounded open set $\Omega\subset\mathbb{R}^n$, $n\geq 2$, a subset $D\subset\Omega$ with Lipschitz boundary, and the conductivity $\sigma(x) := 1+\sigma_D(x)\chi_D(x), \sigma_D \in L_{+}^{\infty}(D)$, then for $1<p<\infty$ the Dirichlet problem for the weighted $p$-Laplace equation can be stated as
\begin{equation}  \label{p-lapPN0}
\begin{cases}
\mdiv(\sigma(x) \abs{\nabla u}^{p-2} \nabla u) = 0 \ \text{in}\ \Omega, \\
u = f \ \text{on}\ \partial\Omega.
\end{cases}
\end{equation}
The problem \eqref{p-lapPN0} is well posed in $W^{1,p}(\Omega)$ for a given Dirichlet boundary data $f\in W^{1,p}(\Omega)$ (the boundary values are understood so that $u-f \in W_0^{1,p}(\Omega)$), see for instance~\cite{D'Onofrio:Iwaniec:2005,Heinonen:Kilpelainen:Martio:1993,Lindqvist:2006,Salo:Zhong:2012}, and the solution $u$ minimizes the $p$-Dirichlet energy 
\[
E_p(v) = \int_{\Omega}\sigma(x)|\nabla v|^pdx
\]
over all $v\in W^{1,p}(\Omega)$ with $v-f \in W^{1,p}_0(\Omega)$. As in the linear case, we formally define the non-linear DN map (keeping the same notations as the linear case), by
\[
\Lambda_{\sigma}(f) := \sigma(x)|\nabla u|^{p-2}\nabla u\cdot \nu|_{\partial\Omega},
\]
where $u\in W^{1,p}(\Omega)$ satisfies \eqref{p-lapPN0}. Using a natural weak definition (see Section \ref{basic}), the DN map becomes a nonlinear map $\Lambda_{\sigma} : X \rightarrow X'$ where $X$ is the abstract trace space $X = W^{1,p}(\Omega)/W_0^{1,p}(\Omega)$ and $X'$ denotes the dual of $X$ (if $\partial \Omega$ has Lipschitz boundary, the trace space $X$ can be identified with the Besov space $B^{1-1/p}_{pp}(\partial \Omega)$). Physically $\Lambda_\sigma (f)$ is the current flux density caused by the boundary potential~$f$. See \cite{Hauer:2014} for further properties of $\Lambda_{\sigma}$.

The precise formulation of the inverse problem studied in this article is as follows.

\bigskip

\textbf{Inverse Problem:} Detect the shape and location of the obstacle $D$ from the knowledge of the nonlinear DN map $\Lambda_{\sigma}$.

\bigskip

As in the linear case, complex geometrical optics type solutions for the $p$-harmonic equation will make it possible to justify the enclosure method. Complex geometrical optics solutions for the $p$-harmonic equation of the form $e^{\rho \cdot x}, \rho\in\mathbb{C}^n,$ with the condition $(p-1)|\mathrm{Re}(\rho)|^2 = |\mathrm{Im}(\rho)|^2$ and $\mathrm{Re}(\rho) \cdot \mathrm{Im}(\rho) = 0,$ were used in~\cite{Salo:Zhong:2012} to prove a boundary determination result for the conductivities. Moreover, certain real valued exponential solutions to the $p$-Laplace equation were introduced by Wolff, see~\cite{Wolff:2007} and also Lemma \ref{lemma:wolff} in Section \ref{CGO1}, and these solutions were used in~\cite{Salo:Zhong:2012} to give a boundary uniqueness result for real valued data.

In order to deal with the enclosure method, both the complex exponentials and Wolff solutions could be used. We will only consider the Wolff solutions in this paper since they lead to a more general result and allow us to detect the obstacle by using real valued boundary data alone. The main components in the proof are a suitable monotonicity inequality, see Lemma \ref{needed_lemma}, and the properties of Wolff solutions. The monotonicity inequality of Lemma \ref{needed_lemma} is a nonlinear version of earlier inequalities in the linear case (see e.g.~\cite[Lemma 1]{Harrach:2013} and references therein). The inequality might also be of interest for other purposes, for example to obtain boundary uniqueness for higher order derivatives of conductivities.

The first contribution to the inverse $p$-Laplace problem is due to Salo and Zhong~\cite{Salo:Zhong:2012}, where a boundary uniqueness result for the conductivities was established by using CGO and Wolff type solutions.
Recently Brander~\cite{Brander:2014} gave a boundary uniqueness result for the first normal derivative of the conductivity.

We also mention the superficially related work of Bolanos and Vernescu~\cite{Bolanos:Vernescu:2014}; they show that one can have ellipsoidal nonlinear inclusions that can be hidden from a single measurement with a layer of 2-harmonic material.
The material in which they embed the inclusions is linear.

Part of the motivation for considering these problems comes from trying to understand inverse boundary value problems for strongly nonlinear equations. The $p$-Laplace type equations are a particular model where CGO type solutions can be used in a genuinely nonlinear way. Beyond the two results mentioned above and the enclosure method established in this paper, there are many open questions for $p$-Laplace type models (including boundary uniqueness for higher order derivatives, interior uniqueness, the validity of sampling or probe type methods, or even the enclosure method for impenetrable obstacles). We also remark that our results include the linear case ($p=2$) as a special case.

The paper is organized as follows: In Section \ref{basic} we prove the monotonicity inequality. Then in Section \ref{CGO1} we discuss the Wolff type solutions and their properties. The statement and proof of the main result are provided in Section \ref{PrfI}. Some further remarks are given in Section~\ref{sec:general}.

\subsection*{Acknowledgements}

All authors were partly supported by the Academy of Finland through the Finnish Centre of Excellence in Inverse Problems Research, and M.K.\ and M.S.\ were also supported in part by an ERC Starting Grant (grant agreement no 307023). We would like to thank the anonymous referee for helpful comments.

\section{Monotonicity inequality}\label{basic}

Let $\Omega \subset \mR^n$ be a bounded open set. If $\sigma \in L^{\infty}_+(\Omega)$, we define the DN map in the weak sense by 
\begin{equation} \label{dnmap_weak_definition}
(\Lambda_{\sigma} f, g) = \int_{\Omega} \sigma \abs{\nabla u}^{p-2} \nabla u \cdot \nabla v \,dx, \qquad f, g \in X,
\end{equation}
where $u \in W^{1,p}(\Omega)$ is the unique solution of $\mdiv(\sigma \abs{\nabla u}^{p-2} \nabla u) = 0$ in $\Omega$ with $u|_{\partial \Omega} = f$, and $v$ is any function in $W^{1,p}(\Omega)$ with $v|_{\partial \Omega} = g$. Recall that $X = W^{1,p}(\Omega)/W_0^{1,p}(\Omega)$, so $(\,\cdot\, , \,\cdot\, )$ is the duality between $X'$ and $X$. See \cite[Appendix]{Salo:Zhong:2012} for more details on the DN map. (Note that in this article we assume all functions are real valued.)

The following monotonicity inequality will be crucial for the enclosure method.
In the linear case $p=2$, this inequality may be found in \cite[Lemma 2.6]{Ide:Isozaki:Nakata:Siltanen:Uhlmann:2007} or \cite[Lemma 1]{Harrach:2013}, see also references therein.

\begin{lemma}\label{needed_lemma}
If $\sigma_0, \sigma_1 \in L^{\infty}_+(\Omega)$ and $1 < p < \infty$, and if $f \in W^{1,p}(\Omega)$, then 
\begin{align}
(p-1) & \int_{\Omega} \frac{\sigma_0}{\sigma_1^{1/(p-1)}} \sulut{\sigma_1^{\frac{1}{p-1}} - \sigma_0^{\frac{1}{p-1}}} \abs{\nabla u_0}^p \,dx \\
& \leq ((\Lambda_{\sigma_1} - \Lambda_{\sigma_0})f, f) 
 \leq \int_{\Omega} (\sigma_1 - \sigma_0) \abs{\nabla u_0}^p \,dx,
\end{align}
where $u_0 \in W^{1,p}(\Omega)$ solves $\mdiv(\sigma_0 \abs{\nabla u_0}^{p-2} \nabla u_0) = 0$ in $\Omega$ with $u_0|_{\partial \Omega} = f$.
\end{lemma}
We emphasise that if $\sigma_1 \geq \sigma_0$, then all the terms in the inequality are nonnegative, but if $\sigma_1 \leq \sigma_0$, they are nonpositive.

\begin{proof}
Let $u_0, u_1 \in W^{1,p}(\Omega)$ be the solutions of the Dirichlet problem for the $p$-Laplace equation,
\begin{equation}  \label{p-lap}
\begin{cases}
\mdiv(\sigma \abs{\nabla u}^{p-2} \nabla u) = 0 \ \text{in}\ \Omega, \\
u = f \ \text{on}\ \partial\Omega,
\end{cases}
\end{equation} 
corresponding to the conductivities $\sigma = \sigma_0$ and $\sigma = \sigma_1$ respectively.

Note that the solution of \eqref{p-lap} can be characterized as the unique minimizer of the energy functional
\[
E(v) = \int_{\Omega}\sigma |\nabla v|^pdx
\]
over the set $\{v\in W^{1,p}(\Omega) ; v-f \in W_{0}^{1,p}(\Omega)\}$ (see \cite[Appendix]{Salo:Zhong:2012}).
Therefore, we obtain the following one sided inequality for the difference of DN maps:
\begin{align*}
((\Lambda_{\sigma_1} - \Lambda_{\sigma_0})f, f) 
& =\int_{\Omega}\sigma_1|\nabla u_1|^pdx - \int_{\Omega}\sigma_0|\nabla u_0|^pdx \\
& \leq \int_{\Omega}(\sigma_1-\sigma_0)|\nabla u_0|^pdx.
\end{align*}

To obtain the other side of the inequality, we rewrite the difference of DN maps as follows:
\begin{align*}
&((\Lambda_{\sigma_1} - \Lambda_{\sigma_0})f, f)  \\
&= \int_{\Omega}\left[\beta\sigma_0|\nabla u_0|^p  - \sulut{(1 + \beta)\sigma_0\abs{\nabla u_0}^{p-2}\nabla u_0 \cdot \nabla u_0  - \sigma_1 \abs{\nabla u_1}^p} \right] \der x \\
&= \int_{\Omega}\left[\beta\sigma_0|\nabla u_0|^p - \left((1+\beta)\sigma_0|\nabla u_0|^{p-2}\nabla u_0\cdot \nabla u_1 - \sigma_1|\nabla u_1|^p\right)\right]dx,
\end{align*}
where $\beta \in \R$; we will later choose $\beta = p-1$.
The last equality holds by the definition of the DN map, since both $u_0$ and $u_1$ have the same Dirichlet boundary values.
Now, by applying Young's inequality $\abs{ab} \leq \frac{\abs{a}^p}{p} + \frac{\abs{b}^{p'}}{p'}$ where $1/p+1/{p'} = 1$ and assuming $\beta \geq -1$, we have 
\begin{align*}
&(1+\beta)\sigma_0|\nabla u_0|^{p-2}\nabla u_0\cdot \nabla u_1 - \sigma_1|\nabla u_1|^p \\
& = \frac{1+\beta}{p^{1/p}}\frac{\sigma_0}{\sigma_{1}^{1/p}}|\nabla u_0|^{p-2}\nabla u_0 \cdot  p^{1/p}\sigma_{1}^{1/p}\nabla u_1 - \sigma_1|\nabla u_1|^p \\
& \leq \frac{1}{p'}\left(\frac{1+\beta}{p^{1/p}}\right)^{p'}\frac{\sigma_{0}^{p'}}{\sigma_{1}^{1/(p-1)}}|\nabla u_0|^p + \sigma_1|\nabla u_1|^p - \sigma_1|\nabla u_1|^p \\
& = \frac{1}{p'}\left(1+\beta\right)^{p'}\frac{1}{p^{1/(p-1)}}\frac{\sigma_{0}^{p'}}{\sigma_{1}^{1/(p-1)}}|\nabla u_0|^p.
\end{align*}

Therefore
\begin{align}
& ((\Lambda_{\sigma_1} - \Lambda_{\sigma_0})f, f) \nonumber \\
& \geq \int_{\Omega}\left(\beta\sigma_0 - \frac{1}{p'}\left(1+\beta\right)^{p'}\frac{1}{p^{1/(p-1)}}\frac{\sigma_{0}^{p'}}{\sigma_{1}^{1/(p-1)}}\right)|\nabla u_0|^pdx \nonumber \\
& = \int_{\Omega}\frac{\beta\sigma_0}{\sigma_{1}^{1/(p-1)}}\left(\sigma_{1}^{\frac{1}{p-1}} - \frac{1}{p'}\frac{(1+\beta)^{p'}}{\beta}\left(\frac{1}{p}\right)^{\frac{1}{p-1}}\sigma_{0}^{\frac{1}{p-1}}\right)|\nabla u_0|^pdx. \label{DNkey}
\end{align}

Note that $\frac{(1+\beta)^{p'}}{\beta} \rightarrow \infty$ as $\beta\rightarrow \infty$ or $\beta\rightarrow 0$. 
So, the function $\beta \rightarrow \frac{(1+\beta)^{p'}}{\beta}$ attains its minimum at $\beta = p-1$.
Thus, we choose $\beta = p-1$ so that from \eqref{DNkey}, we obtain the required inequality. Note that choosing $\beta=p-1$ in the beginning of the proof would have simplified the constants in the argument.
\end{proof}

\section{Wolff solutions}\label{CGO1}

The other main tool in justifying the enclosure method is the use of appropriate exponentially growing solutions. In the following lemma we describe real valued exponential solutions of the $p$-Laplace equation.
The solutions are periodic in one direction and behave exponentially in a perpendicular direction.
They were first introduced by Wolff~\cite[section 3]{Wolff:2007} and later applied to inverse problems in~\cite[section 3]{Salo:Zhong:2012}.

\begin{lemma}\label{lemma:wolff}
Let $\rho, \rho^{\perp} \in \mR^n$ satisfy $\abs{\rho} = \abs{\rho^{\perp}} = 1$ and $\rho \cdot \rho^{\perp} = 0$. Define $h \colon \R^n \to \R$ by $h(x) = e^{-\rho \cdot x}a(\rho^\perp \cdot x)$, where the function~$a$ satisfies the differential equation
\begin{equation}\label{eq:wolff}
a''(s) + V(a,a')a = 0
\end{equation}
with
\begin{equation}
V(a,a') = \frac{(2p-3)\left(a'\right)^2+(p-1)a^2}{(p-1)\left(a'\right)^2 + a^2},
\end{equation}
The function $h$ is then $p$-harmonic.

Given any initial conditions $(a_0,b_0) \in \R^2 \setminus \{(0,0)\}$ there exists a solution $a \in C^\infty(\R)$ to the differential equation~\eqref{eq:wolff} which is periodic with period $\lambda_p > 0$, satisfies the initial conditions $(a(0),a'(0)) = (a_0,b_0)$, satisfies $\int_0^{\lambda_p} a(s) \,ds = 0$, and furthermore there exist constants $c$ and $C$ depending on $a_0,b_0,p$ such that for all $s \in \R$ we have
\begin{equation}\label{eq:wolff_new}
C > a(s)^2+a'(s)^2 > c > 0.
\end{equation}
\end{lemma}
\begin{proof}
Since the $p$-Laplace operator is rotation and reflection invariant, we can take $\rho = e_1 = (1,0,\ldots,0)$ and $\rho^\perp = e_n = (0,\ldots,0,1)$.

Almost all of the claimed results are explicitly stated and proved in~\cite[Lemma 3.1]{Salo:Zhong:2012}; we only need to check that the claim~\eqref{eq:wolff_new} holds.
The upper bound follows from smoothness and periodicity of the function $a$.
The lower bound is also proven in~\cite[Lemma 3.1]{Salo:Zhong:2012}, though not explicitly mentioned in the statement of the lemma. The lemma proves that if $I$ is the maximal interval of existence of the solution $(a,a'): I \to \mR^2 \setminus \{(0,0)\}$ of the ODE, then in fact $I = \mR$, so one has $(a(s),a'(s)) \neq (0,0)$ for all $s \in \mR$. Then smoothness and periodicity imply that $a(s)^2 + a'(s)^2 \geq c > 0$ for all $s \in \mR$.
\end{proof}

Let $\tau \in \R$ be a large parameter and $t \in \R$ be a constant.
With the notation of Lemma \ref{lemma:wolff}, define $u_0 \colon \R^n \to \R$ by
\begin{equation}\label{eq:real_test}
u_0(x) = e^{\tau(x \cdot \rho - t)} a\left(\tau x \cdot \rho^{\perp}\right).
\end{equation}
We use a fixed function $a$ (for some fixed initial data $(a_0,b_0) \neq (0,0)$) and fixed directions $\rho$ and $\rho^\perp$ throughout the article.
By Lemma \ref{lemma:wolff}, $u_0$ satisfies the $p$-Laplace equation in $\mR^n$. Note that $a$ is oscillating as a periodic function which integrates to zero over the period. Thus $u_0$ has exponential behavior in the $\rho$ direction and oscillates rapidly in the $\rho^{\perp}$ direction if $\tau$ is large. The functions $u_0$ will be used as the complex geometrical optics solutions in the enclosure method.

We record a formula for the gradient of $u_0$, which will be used later:
\begin{align}
\nabla u_0 = \tau \expo\sulut{\tau(x \cdot \rho - t)}\sulut{\rho a\sulut{\tau x \cdot \rho^\perp} + \rho^\perp a'\sulut{\tau x \cdot \rho^\perp}}.
\label{eq:nablap}
\end{align}
Note that
\begin{equation}
0 < c < \abs{\sulut{\rho a\sulut{\tau x \cdot \rho^\perp} + \rho^\perp a'\sulut{\tau x \cdot \rho^\perp}}}^2 < C
\end{equation}
for all $x \in \mR^n$ by equation~\eqref{eq:wolff_new}.

\section{Main result and the proof}\label{PrfI}
In this section we use the following standing assumptions, unless otherwise mentioned.
Also, we use $u_0$ to denote the solutions \eqref{eq:real_test}, and use them in the definition of the indicator and support functions below.

Recall that we consider the set $\Omega\subset\mathbb{R}^n, n\geq 2$, to be a bounded domain.
The inclusion $D$, $D\subset\Omega$, is assumed to be a bounded open set with Lipschitz boundary.
We furthermore assume that the conductivity $\sigma$ has a jump discontinuity along the interface~$\doo D$.
In particular, we assume that $\sigma(x) := 1 + \sigma_D(x)\chi_D(x)$, where $\sigma_D\in L_+^{\infty}(D)$ and $\chi_D$ is the characteristic function of $D$.
We let $1 < p < \infty$ and consider the following Dirichlet problem for $f \in W^{1,p}(\Omega)$:
\begin{equation}  \label{p-lapPN}
\begin{cases}
\mdiv(\sigma(x) \abs{\nabla u}^{p-2} \nabla u) = 0 \ \text{in}\ \Omega, \\
u = f \ \text{on}\ \partial\Omega.
\end{cases}
\end{equation}
 
\begin{definition}[Indicator function]
We define the indicator function
\[
I_{\rho}(u_0,\tau,t) := \tau^{(n-p)} ((\Lambda_{\sigma}-\Lambda_{1})(u_0),u_0)
\]
where $\rho \in S^{n-1}$, $u_0$ is the Wolff solution~\eqref{eq:real_test}, $\tau > 0$, $t \in \mR$, and $\Lambda_{\sigma}$ and $\Lambda_{1}$ are the non-linear DN maps corresponding to the conductivities $\sigma$, $1$ respectively, defined by \eqref{dnmap_weak_definition}.
\end{definition}

\begin{definition}[Support function]
We define the support function $h_D(\rho)$ of $D$ as
\[
h_D(\rho) := \sup_{x\in D}x\cdot\rho, \qquad \rho \in S^{n -1}.
\]
\end{definition}

Now we state our main result.
\begin{thm}\label{main_thm}
Given the standing assumptions (see the beginning of Section~\ref{PrfI}) we have the following characterization of $h_{D}(\rho)$.
\begin{enumerate} 
\item When $t>h_D(\rho)$ we have
\begin{equation}\lim_{\tau \to \infty}\vert I_{\rho}(u_0,\tau,t)\vert = 0,\end{equation}
and more precisely,
\begin{equation}\label{main_behavior_1far}\vert I_{\rho}(u_0,\tau,t)\vert \leq Ce^{-c\tau}\end{equation}
for $\tau > 0$, and where $C,c>0$.
\item When $t = h_D(\rho)$ we have
\begin{equation}\liminf_{\tau \to \infty}|I_{\rho}(u_0,\tau, h_{D}(\rho))| >0,\end{equation}
and more precisely
\begin{equation}\label{main_behavior_2far}
 c \leq  \vert I_{\rho}(u_0,\tau,h_D(\rho)) \vert \leq C \tau^n
\end{equation}
when $c, C>0$, and $\tau >0$ for the upper bound, and $\tau \gg 1$ for the lower bound.
\item When $t < h_D(\rho)$ we have
\begin{equation}
\lim_{\tau \to \infty}|I_{\rho}(u_0,\tau,t)| = \infty,
\end{equation}
and more precisely,
\begin{equation}\label{main_behavior_3far}
\vert I_{\rho}(u_0,\tau,t) \vert \geq ce^{c\tau},
\end{equation}
where $\tau \gg1$, $c>0$.
\end{enumerate}
\end{thm}
From this theorem, we see that, for a fixed direction $\rho,$ the behavior of the indicator function $I_{\rho}(u_0,\tau,t)$ changes drastically in terms of $\tau.$ 
Precisely, for $t>h_D(\rho)$ it is decaying exponentially, for $t<h_D(\rho)$ it is growing exponentially and for $t=h_D(\rho)$ 
it has a polynomial behavior. Using this property of the indicator function and varying $\rho$, we can reconstruct the support function $h_{D}(\rho)$ from the non-linear DN map.

The proof consists of the following steps: In lemmata~\ref{lemma:scpmm} and \ref{lemma:reduction} we show that proving part (2) of the main theorem is enough.
Lemmata~\ref{lemma:upper_bound} and \ref{pinf} complete the proof by showing that the upper bound and the lower bound hold.

\begin{lemma} \label{lemma:scpmm}
Given the standing assumptions (see the beginning of Section~\ref{PrfI}),
the identity
\begin{equation}\label{scpmm}
  I_{\rho}(u_0,\tau,t) = e^{2\tau(h_{D}(\rho)-t)} I_{\rho}(u_0,\tau,h_{D}(\rho)),
\end{equation}
holds.
\end{lemma}
The lemma is a straightforward consequence of the definitions of the indicator function and the Wolff solutions.
In particular, we do not need any assumptions on the inclusion~$D$.

Another way of reconstructing the support function from Theorem \ref{main_thm} is as follows.
\begin{lemma}
Given the standing assumptions (see the beginning of Section~\ref{PrfI}),
we have the formula
\begin{equation}
h_{D}(\rho)-t= \lim_{\tau \to \infty}
\frac{\log|I_{\rho}(u_0,\tau,t)|}{2\tau}
\end{equation}
\end{lemma}
\begin{proof}
Follows from \eqref{main_behavior_2far} and the identity~\eqref{scpmm}.
\end{proof}
Finally, from this support function we can estimate the convex hull of $D$, since for every direction $\rho$ the support function determines a half-space that must contain the inclusion $D$, and so that boundary of the half-space intersects $\doo D$.
The intersection of such half-spaces is the convex hull of the obstacle~$D$.
Thus we obtain
\begin{corollary}
From the knowledge of the DN map we can recover the convex hull of the inclusion~$D$.
\end{corollary}
If the domain $\Omega$ is not connected, then we can consider it component-wise by using a test function which equals the Wolff solution in the component under investigation and vanishes elsewhere.
Hence, we can recover the convex hull of the inclusion within any fixed component.

We will now prove Theorem \ref{main_thm}. To do this, we need only to show the estimate~\eqref{main_behavior_2far} as the other properties \eqref{main_behavior_1far} and \eqref{main_behavior_3far} will follow from \eqref{main_behavior_2far} and the identity \eqref{scpmm}, as stated in the next lemma:
\begin{lemma} \label{lemma:reduction}
Given the standing assumptions (see the beginning of Section~\ref{PrfI}),
if we have  $c \leq  \abs{I_{\rho}(u_0,\tau,h_D(\rho))} \leq C \tau^n$ for $\tau \gg 1$ and some positive constants $c,C$,
then Theorem~\ref{main_thm} holds.
\end{lemma}
Actually, we only need $\tau > 0$ for the upper bound.
This is clear from the proof of the next lemma.
\begin{lemma} \label{lemma:upper_bound}
Given the standing assumptions (see the beginning of Section~\ref{PrfI}),
we have the upper bound
\begin{equation}
\abs{I_{\rho}(u_0,\tau,h_D(\rho))} \leq C \tau^n
\end{equation}
in equation~\eqref{main_behavior_2far}.
\end{lemma}
\begin{proof}
We have by Lemma~\ref{needed_lemma}
\begin{equation} \label{eq:lower_proof}
\abs{I_{\rho}(u_0,\tau,h_D(\rho))} \leq \tau^{n-p} \int_{\Omega} (\sigma - 1) \abs{\nabla u_0}^p \der x.
\end{equation}
By \eqref{eq:nablap} we have
\begin{equation}
\abs{\nabla u_0}^p \leq C \tau^p,
\end{equation}
since $t = h_D(\rho)$.
By assumption $\sigma > 1+\eps$ on a set $D$ of positive measure.
This completes the proof.
\end{proof}

We do not need any geometric assumptions on the inclusion~$D$ for the previous lemma (positive measure is sufficient). The proof of the lower bound in \eqref{main_behavior_2far} is more difficult.

Now, from Lemma~\ref{needed_lemma} we can write
\begin{align} \notag
|I_{\rho}(u_0, \tau, h_D(\rho))|
& \geq \tau^{(n-p)}\int_{\Omega}(p-1)\frac{1}{\sigma^{\frac{1}{p-1}}}(\sigma^{\frac{1}{p-1}}-1)|\nabla u_0|^p dx \\
& \geq C \tau^{(n-p)}\int_D |\nabla u_0|^p dx.
\label{eq:lower_bound}
\end{align}

By \eqref{eq:nablap} we obtain at $t = h_D(\rho)$
\begin{align}
|I_{\rho}(u_0, \tau, h_D(\rho))|
& \geq C \tau^{(n-p)}\int_D \tau^p e^{-p\tau(h_D(\rho)-x \cdot\rho)} dx \nonumber \\
& \geq C \tau^n \int_D e^{-p\tau(h_D(\rho)- x\cdot\rho)}dx. \label{penEs}
\end{align}

In preparation for the next lemma we define a cover of the set $K := {\partial D \cap\{x\cdot\rho = h_{D}(\rho)\}}$.
For any $\alpha \in K$ and $\delta >0$
we define
\begin{equation}B(\alpha,\delta) := \{x\in \mathbb{R}^n ; | x- \alpha |<\delta\}.\end{equation}
Then, $K \subset \cup_{\alpha\in K}B(\alpha ,\delta)$.
Since $K$ is compact, 
there exist $\alpha_1,\cdots,\alpha_N \in K$ such that $ K \subset {B(\alpha_{1},\delta)\cup \cdots \cup B(\alpha_{N},\delta)}$. 
We define 
$
 D_{j,\delta}:= D \cap B(\alpha_{j},\delta), D_{\delta}:= \cup_{j=1}^{N}D_{j,\delta}.
$
Also
\begin{equation}
 \int_{D \setminus D_ \delta} {e^{-p \tau (h_D(\rho)-x \cdot \rho)}} dx = O(e^{-pc \tau})
\end{equation}
with positive constant $c$ as $\tau \rightarrow \infty$.
Fix $j \in \{1,\ldots,N\}$, so that $\alpha_j \in K$.
By translation we may assume that $\alpha_j = 0$.
Then there exists a hyperplane $H_j$ so that we can parametrise the boundary $\doo D$ near $\alpha_j$ as a Lipschitz function defined on $H_j$, and with $\rho \notin H_j$.%
\footnote{To find this hyperplane, take a vector $w$ that points inside $D$.
Then $w \cdot \rho < 0$ and we can take $H_j$ to be the orthocomplement of the line spanned by $w$.}
Let $y' = (y_1,\ldots,y_{n-1})$ be a basis of $H_j$.
Denote the parametrization of $\partial D$ near $\alpha_j$
by $l_j(y')$.
We select the remaining coordinate $y_n$ from the linear subspace spanned by $w$ so that we have
\begin{equation}\label{3.27}
y_n = h_D(\rho) - x \cdot \rho.
\end{equation}

\begin{lemma}\label{pinf}
We assume the standing assumptions (see the beginning of Section~\ref{PrfI}).
For $1<p<\infty$, the following estimate holds for $\tau \gg 1$:
\[
\int_D e^{-p\tau(h_D(\rho)-x \cdot\rho)}dx \geq C\tau^{-n}.
\]
\end{lemma}
\begin{proof}
\begin{align}
\int_D e^{-p\tau(h_D(\rho)-x \cdot\rho)}dx 
& \geq \int_{D_{\delta}} e^{-p\tau(h_{D}(\rho)-x \cdot\rho)}dx \\
& \geq C_{\delta,N} \sum_{j=1}^{N}\int_{|y'|<\delta} dy' \int_{l_j(y')}^{\delta} e^{-p\tau y_n}dy_n \nonumber \\
& \geq C \tau^{-1} \sum_{j=1}^{N}\int_{|y'|<\delta} e^{-p\tau l_j(y')} dy' - \frac{C}{p}\tau^{-1}e^{-p\delta\tau} \label{esii}.
\end{align}
Since we have $ l_j(y') \leq C | y'|$ as $ \partial D$ is Lipschitz,
we have the following estimate:
\begin{align}
\sum_{j=1}^N \int_{|y'|<\delta} e^{-p\tau l_j(y')} dy' \label{eq:lip_estimate}
&\geq C\sum_{j=1}^N \int_{|y'|<\delta} e^{-C p\tau |y'|} dy' \\
&\geq C\tau^{-(n-1)} \sum_{j=1}^N \int_{|y'|<\tau \delta} e^{-C p |y'|} dy'
\geq C \tau^{1-n}. \label{lower_v}
\end{align}
Hence, combining the estimates \eqref{esii} and \eqref{lower_v}, we obtain the required estimate for \mbox{$\tau \gg 1$.}
\end{proof}
Finally, the proof of Theorem \ref{main_thm} for the penetrable obstacle case follows from equation~\eqref{penEs} and lemma~\ref{pinf}.

\section{Remarks}\label{sec:general}
In this section we give some further remarks on the problem. These remarks refer to the proofs in Section~\ref{PrfI}.
\begin{rem}\label{remark:sigma_small}
We can alternatively have an inclusion with smaller conductivity than its surroundings, which corresponds to $0 < \eps < \sigma_D < 1-\eps$ for some $\eps > 0$.
The only difference in proofs is that the inequalities in the monotonicity inequality, Lemma~\ref{needed_lemma}, have reversed roles in the following proofs:
\begin{itemize}
\item At the very beginning of the proof of Lemma~\ref{lemma:upper_bound}.
\item In the inequality~\eqref{eq:lower_bound}.
\end{itemize}
The sign of the indicator function is negative when the inclusion has smaller conductivity than the background and positive when the inclusion has higher conductivity than the background, as is easy to see from the monotonicity inequality (Lemma~\ref{needed_lemma}).
\end{rem}
We can say something even when there are regions of higher and regions of lower conductivity.
Let us write $D_- = \{ x \in \Omega; \sigma(x) < 1-c \}$ and $D_+ = \{ x \in \Omega; \sigma(x) > 1+c \}$ for some $c>0$ with the assumption that $\sigma(x) = 1$ in the complement of $D = D_+ \cup D_-$.
We also assume that $D_\pm$ both satisfy the assumptions we have previously made of $D$.
Consider a direction $\rho \in S^{n-1}$ such that $\argmax_{x \in \ol{D}} x \cdot \rho \subset \sulut{\ol{D_+} \setminus \ol{D_-}}$, where $\argmax_x A(x)$ is the set of $x$ where $A(x)$ reaches its maximum.
Then there is $\eps > 0$ with
\begin{equation}
\dist\sulut{\aaltosulut{x \in \Omega; x \cdot \rho = h_D(\rho)},D_-} = \eps,
\end{equation}
and we have the following estimate for $x \in D_-$:
\begin{equation}
\abs{\nabla u(x)} \leq C \tau \expo\sulut{\tau \sulut{x \cdot \rho - h_D(\rho)}} \leq C \tau \expo\sulut{-\eps \tau} \to 0
\end{equation}
as $\tau \to \infty$.
Note that the argument is still valid if we change the roles of $D_+$ and $D_-$.

We are about to check that the upper and lower bounds in the estimate~\eqref{main_behavior_2far} hold in this more complicated situation.
For the upper bound we have an additional exponentially small term in the inequality~\eqref{eq:lower_proof}; we can absorb it into the constant by assuming that $\tau \geq 1$.
For the lower bound a similar error term first appears in estimate~\eqref{eq:lower_bound}.
One can remove the error term by selecting a sufficiently small $\delta > 0$ later in the proof, so that $D_- \subset \sulut{D \setminus D_\delta}$

\begin{figure}
\centering
\includegraphics[width=0.4\textwidth]{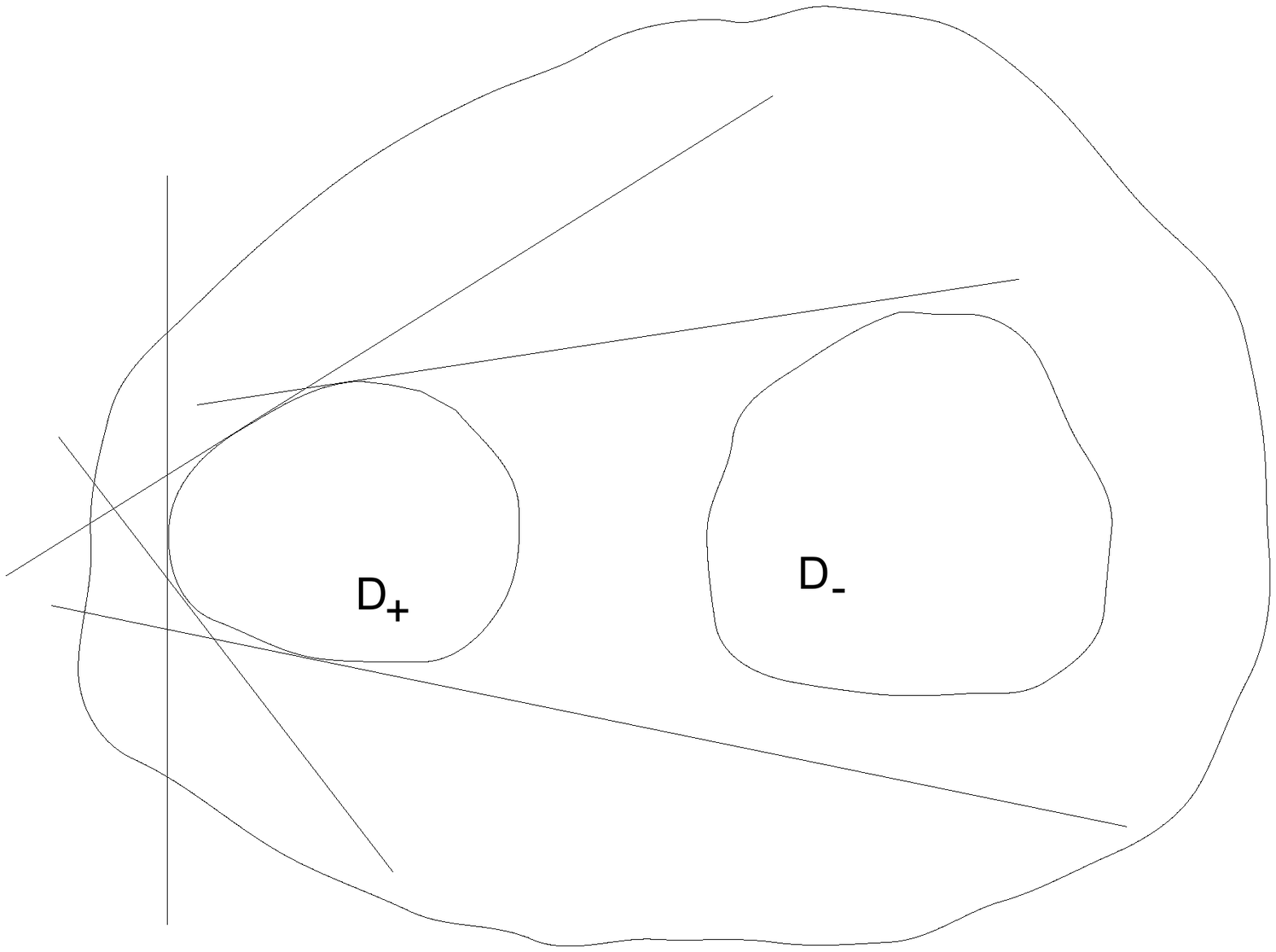}
\caption{The left half of boundary of the domain $D_{+}$ can be detected.}
\label{figure:good}
\end{figure}

\begin{figure}
\centering
\includegraphics[width=0.6\textwidth]{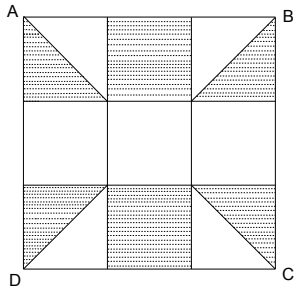}
\caption{The interior of the region ABCD, the dashed regions are denoted by $D_{+}$ and the other parts are denoted by $D_{-}$.}
\label{figure:chess}
\end{figure}

In this way, if we know that there is a direction from which we will first hit either $D_+$ or $D_-$, but not both, then we can identify when we hit the boundary of the obstacle and whether we hit $D_+$ or $D_-$ (from the sign of the indicator function).
We need to know the good directions a priori, or they need to have full measure.
If the good directions form a dense subset of $S^{n-1}$, then we can recover the convex hull of $D$ and recognise some boundary points as boundary points of either $\co (D_+)$ or $\co(D_-)$.
If there are fewer good directions, then we might only be able to enclose a larger set then $\co D$.
As an illustrative example consider the simple case where $D_\pm$ are nice convex sets; see figure~\ref{figure:good}.
For a counterexample see the chessboard figure~\ref{figure:chess}, where $D_\pm$ are finite unions of convex sets with Lipschitz boundaries, but there are no good directions.

\begin{rem}
The inclusion $D$ does not need to have Lipschitz boundary everywhere. The regularity of the boundary $\doo D$ is only used in proving Lemma~\ref{pinf}.

As a first observation, we only consider the points of $\doo D$ that are also boundary points of the convex hull~$\co D$, so we only need to impose restrictions there.
In particular, the set~$D$ does not need to be open near points that are not near $\doo \sulut{\co D}$.
As a second observation, we need to be able to parametrise $\doo D$ near every point~$x_0 \in \doo D \cap \doo \sulut{\co D}$ as a function defined on some hyperspace $H$, so that $\rho \notin H$.
To have the estimate~\eqref{eq:lip_estimate} it is sufficient to have Lipschitz boundary near those points.
\end{rem}

We also get partial results even when we do not have an inclusion, but do have monotonicity:
\begin{rem}
Suppose $\sigma \geq 1$ and the inequality is strict on a set of positive measure.
Write $D = \left\{ x \in \Omega; \sigma(x) > 1 \right\}$.
Then we have lemma~\ref{lemma:upper_bound}.
By the identity~\eqref{scpmm} we know that for $t > h_D(\rho)$
\begin{equation}
\vert I_{\rho}(u_0,\tau,t)\vert \leq Ce^{-c\tau}
\end{equation}
for $\tau > 0$, and with $C>0$.
\end{rem}
That is, we know what must happen if $t > h_D(\rho)$.
If the absolute value of the indicator function does not vanish in the limit, then we must have $t \leq h_D(\rho)$.
This is a sufficient condition, so it finds a subset of $\co D$, which might be empty.
The remark also applies to $\sigma \leq 1$; see remark~\ref{remark:sigma_small}.

\bibliographystyle{abbrv}
\bibliography{math}

\begin{thebibliography}{10}

\bibitem{Ammari:Bonnetier:Capdebocsg:Fink:Tanter:2008}
H.~Ammari, E.~Bonnetier, Y.~Capdeboscq, M.~Tanter, and M.~Fink.
\newblock {Electrical Impedance Tomography by Elastic Deformation}.
\newblock {\em SIAM Journal on Applied Mathematics}, 68(6):1557--1573, June
  2008.

\bibitem{Antontsev:Rodrigues:2006}
S.~N. Antontsev and J.~F. Rodrigues.
\newblock On stationary thermo-rheological viscous flows.
\newblock {\em Annali dell'Universita di Ferrara}, 52(1):19--36, 2006.

\bibitem{Aronsson:1996}
G.~Aronsson.
\newblock On $p$-harmonic functions, convex duality and an asymptotic formula
  for injection mould filling.
\newblock {\em European Journal of Applied Mathematics}, 7:417--437, Oct. 1996.

\bibitem{Aronsson:Janfalk:1992}
G.~Aronsson and U.~Janfalk.
\newblock On {H}ele-{S}haw flow of power-law fluids.
\newblock {\em European Journal of Applied Mathematics}, 3:343--366, Dec. 1992.

\bibitem{Arridge:Scherzer:2012}
S.~R. Arridge and O.~Scherzer.
\newblock Imaging from coupled physics.
\newblock {\em Inverse Problems}, 28(8):080201, 2012.

\bibitem{Atkinson:Champion:1984}
C.~Atkinson and C.~R. Champion.
\newblock Some boundary-value problems for the equation {$\nabla \cdot
  \sulut{\abs{\nabla \phi}^N \nabla \phi}$}.
\newblock {\em The Quarterly Journal of Mechanics and Applied Mathematics},
  37(3):401--419, 1984.

\bibitem{Bal:2012}
G.~Bal.
\newblock {Cauchy problem for Ultrasound Modulated EIT}.
\newblock {\em Analysis \& PDE}, 6(4):751--775, 2013.

\bibitem{Bal:Schotland:2010}
G.~Bal and J.~C. Schotland.
\newblock Inverse scattering and acousto-optic imaging.
\newblock {\em Physical Review Letters}, 104(4), 2010.

\bibitem{Berselli:Diening:Ruzicka:2010}
L.~C. Berselli, L.~Diening, and M.~R\r{u}\v{z}i\v{c}ka.
\newblock Existence of strong solutions for incompressible fluids with shear
  dependent viscosities.
\newblock {\em Journal of Mathematical Fluid Mechanics}, 12(1):101--132, 2010.

\bibitem{Bolanos:Vernescu:2014}
S.~J. Bola\~nos and B.~Vernescu.
\newblock Nonlinear neutral inclusions: Assemblages of coated ellipsoids.
\newblock {\em {arXiv} preprints}, Sept. 2014.
\newblock Available online
  \href{http://arxiv.org/abs/1409.4786}{ar{X}iv:1409.4786}.

\bibitem{Brander:2014}
T.~Brander.
\newblock Calder\'on problem for the $p$-{L}aplacian: First order derivative of
  conductivity on the boundary.
\newblock {\em ArXiv e-prints}, Mar. 2014.
\newblock To appear in Proc. AMS. Available online at
  \href{http://arxiv.org/abs/1403.0428}{ar{X}iv:1403.0428}.

\bibitem{Bueno:Longo:Varela:2008}
P.~R. Bueno, J.~A. Varela, and E.~Longo.
\newblock {SnO}$_2$, {ZnO} and related polycrystalline compound semiconductors:
  An overview and review on the voltage-dependent resistance (non-ohmic)
  feature.
\newblock {\em Journal of the European Ceramic Society}, 28(3):505 -- 529,
  2008.

\bibitem{Cakoni:Colton:2006}
F.~Cakoni and D.~Colton.
\newblock {\em Qualitative methods in inverse scattering theory}.
\newblock Interaction of Mechanics and Mathematics. Springer-Verlag, Berlin,
  2006.
\newblock An introduction.

\bibitem{Church:Gagnon:McFee:Wort:2006}
P.~Church, J.~McFee, S.~Gagnon, and P.~Wort.
\newblock Electrical impedance tomographic imaging of buried landmines.
\newblock {\em Geoscience and Remote Sensing, IEEE Transactions on},
  44(9):2407--2420, Sept. 2006.

\bibitem{Colton:Kirsch:1996}
D.~Colton and A.~Kirsch.
\newblock A simple method for solving inverse scattering problems in the
  resonance region.
\newblock {\em Inverse Problems}, 12(4):383--393, 1996.

\bibitem{D'Onofrio:Iwaniec:2005}
L.~D'Onofrio and T.~Iwaniec.
\newblock Notes on {$p$}-harmonic analysis.
\newblock In {\em The {$p$}-harmonic equation and recent advances in analysis},
  volume 370 of {\em Contemporary Mathematics}, pages 25--50. American
  Mathematical Society, Providence, RI, 2005.

\bibitem{Evans:2007}
L.~C. Evans.
\newblock The 1-{L}aplacian, the $\infty$-{L}aplacian and differential games.
\newblock {\em Perspect. Nonlinear Partial Differ. Equ.: In Honor of Haim
  Brezis}, 446:245, 2007.

\bibitem{Garroni:Kohn:2003}
A.~Garroni and R.~V. Kohn.
\newblock Some three--dimensional problems related to dielectric breakdown and
  polycrystal plasticity.
\newblock {\em Proceedings of the Royal Society of London. Series A:
  Mathematical, Physical and Engineering Sciences}, 459(2038):2613--2625, 2003.

\bibitem{Garroni:Nesi:Ponsiglione:2001}
A.~Garroni, V.~Nesi, and M.~Ponsiglione.
\newblock Dielectric breakdown: optimal bounds.
\newblock {\em Proceedings of the Royal Society of London. Series A:
  Mathematical, Physical and Engineering Sciences}, 457(2014):2317--2335, 2001.

\bibitem{Gebauer:Scherzer:2008}
B.~Gebauer and O.~Scherzer.
\newblock Impedance-acoustic tomography.
\newblock {\em SIAM Journal on Applied Mathematics}, 69(2):565--576, 2008.

\bibitem{Glowinski:Rappaz:2003}
R.~Glowinski and J.~Rappaz.
\newblock Approximation of a nonlinear elliptic problem arising in a
  non-{N}ewtonian fluid flow model in glaciology.
\newblock {\em ESAIM: Mathematical Modelling and Numerical Analysis},
  37:175--186, 1 2003.

\bibitem{Harrach:2013}
B.~Harrach.
\newblock Recent progress on the factorization method for electrical impedance
  tomography.
\newblock {\em Comput. Math. Methods Med.}, 2013:Art. ID 425184, 8, 2013.

\bibitem{Harrach:Ullrich:2013}
B.~Harrach and M.~Ullrich.
\newblock Monotonicity-based shape reconstruction in electrical impedance
  tomography.
\newblock {\em SIAM J. Math. Anal.}, 45(6):3382--3403, 2013.

\bibitem{Hauer:2014}
D.~Hauer.
\newblock The $p$-{D}irichlet-to-{N}eumann operator with applications to
  elliptic and parabolic problems.
\newblock Mar. 2014.
\newblock Preprint, retrieved 18.9.2014 from
  \url{www.maths.usyd.edu.au/u/pubs/publist/preprints/2014/hauer-9.html}.

\bibitem{Heinonen:Kilpelainen:Martio:1993}
J.~Heinonen, T.~Kilpel{\"a}inen, and O.~Martio.
\newblock {\em Nonlinear potential theory of degenerate elliptic equations}.
\newblock Oxford Mathematical Monographs. The Clarendon Press, Oxford
  University Press, New York, Oxford, 1993.
\newblock Oxford Science Publications.

\bibitem{Hoell:Moradifan:Nachman:2013}
N.~Hoell, A.~Moradifam, and A.~Nachman.
\newblock Current density impedance imaging of an anisotropic conductivity in a
  known conformal class.
\newblock {\em SIAM J. Math. Anal.}, 46:1820--1842, 2014.

\bibitem{Ide:Isozaki:Nakata:Siltanen:Uhlmann:2007}
T.~Ide, H.~Isozaki, S.~Nakata, S.~Siltanen, and G.~Uhlmann.
\newblock Probing for electrical inclusions with complex spherical waves.
\newblock {\em Comm. Pure Appl. Math.}, 60(10):1415--1442, 2007.

\bibitem{Idiart:2008}
M.~I. Idiart.
\newblock The macroscopic behavior of power-law and ideally plastic materials
  with elliptical distribution of porosity.
\newblock {\em Mechanics Research Communications}, 35(8):583--588, 2008.

\bibitem{Ikehata:1998}
M.~Ikehata.
\newblock Reconstruction of the shape of the inclusion by boundary
  measurements.
\newblock {\em Comm. Partial Differential Equations}, 23(7-8):1459--1474, 1998.

\bibitem{Ikehata:1999}
M.~Ikehata.
\newblock How to draw a picture of an unknown inclusion from boundary
  measurements. {T}wo mathematical inversion algorithms.
\newblock {\em J. Inverse Ill-Posed Probl.}, 7(3):255--271, 1999.

\bibitem{Ikehata:2000}
M.~Ikehata.
\newblock On reconstruction in the inverse conductivity problem with one
  measurement.
\newblock {\em Inverse Problems}, 16(3):785--793, 2000.

\bibitem{Ikehata:2002}
M.~Ikehata.
\newblock Reconstruction of inclusion from boundary measurements.
\newblock {\em J. Inverse Ill-Posed Probl.}, 10(1):37--65, 2002.

\bibitem{Ikehata:2010}
M.~{Ikehata}.
\newblock {The probe and enclosure methods for inverse obstacle scattering
  problems. The past and present}.
\newblock {\em RIMS K{\^o}ky{\^u}roku}, 1702:1--22, 2010.

\bibitem{Isakov:1988}
V.~Isakov.
\newblock On uniqueness of recovery of a discontinuous conductivity
  coefficient.
\newblock {\em Comm. Pure Appl. Math.}, 41(7):865--877, 1988.

\bibitem{Henkelman:Joy:Scott:1989}
M.~Joy, G.~Scott, and M.~Henkelman.
\newblock In vivo detection of applied electric currents by magnetic resonance
  imaging.
\newblock {\em Magnetic Resonance Imaging}, 7(1):89 -- 94, 1989.

\bibitem{Kar:Sini:2014}
M.~Kar and M.~Sini.
\newblock Reconstruction of interfaces using {CGO} solutions for the {M}axwell
  equations.
\newblock {\em J. Inverse Ill-Posed Probl.}, 22(2):169--208, 2014.

\bibitem{Kim:Kwon:Seo:Yoon:2002}
S.~Kim, O.~Kwon, J.~Seo, and J.~Yoon.
\newblock On a nonlinear partial differential equation arising in magnetic
  resonance electrical impedance tomography.
\newblock {\em SIAM Journal on Mathematical Analysis}, 34(3):511--526, 2002.

\bibitem{Kirsch:Grindberg:2008}
A.~Kirsch and N.~Grinberg.
\newblock {\em The factorization method for inverse problems}, volume~36 of
  {\em Oxford Lecture Series in Mathematics and its Applications}.
\newblock Oxford University Press, Oxford, 2008.

\bibitem{Kohn:Levy:1998}
O.~Levy and R.~V. Kohn.
\newblock Duality relations for non-{O}hmic composites, with applications to
  behavior near percolation.
\newblock {\em Journal of Statistical Physics}, 90(1--2):159--189, 1998.

\bibitem{Liimatainen:Salo:2012}
T.~Liimatainen and M.~Salo.
\newblock {$n$}-harmonic coordinates and the regularity of conformal mappings.
\newblock {\em Math. Res. Lett.}, 21:341--361, 2014.

\bibitem{Lindqvist:2006}
P.~Lindqvist.
\newblock {\em Notes on the {$p$-L}aplace equation}, volume 102 of {\em Reports
  of {U}niversity of {J}yv{\"a}skyl{\"a} Department of Mathematics and
  Statistics}.
\newblock University of {J}yv{\"a}skyl{\"a}, Jyv{\"a}skyl{\"a}, Finland, 2006.

\bibitem{Nachman:Tamasan:Timonov:2009}
A.~Nachman, A.~Tamasan, and A.~Timonov.
\newblock Recovering the conductivity from a single measurement of interior
  data.
\newblock {\em Inverse Problems}, 25(3):035014, 2009.

\bibitem{Nagayasu:Uhlmann:Wang:2011}
S.~Nagayasu, G.~Uhlmann, and J.-N. Wang.
\newblock Reconstruction of penetrable obstacles in acoustic scattering.
\newblock {\em SIAM J. Math. Anal.}, 43(1):189--211, 2011.

\bibitem{Nakamura:Uhlmann:Wang:2005}
G.~Nakamura, G.~Uhlmann, and J.-N. Wang.
\newblock Oscillating-decaying solutions, {R}unge approximation property for
  the anisotropic elasticity system and their applications to inverse problems.
\newblock {\em J. Math. Pures Appl. (9)}, 84(1):21--54, 2005.

\bibitem{Nakamura:Yoshida:2007}
G.~Nakamura and K.~Yoshida.
\newblock Identification of a non-convex obstacle for acoustical scattering.
\newblock {\em J. Inverse Ill-Posed Probl.}, 15(6):611--624, 2007.

\bibitem{Parker:1984}
R.~L. Parker.
\newblock The inverse problem of resistivity sounding.
\newblock {\em GEOPHYSICS}, 49(12):2143--2158, 1984.

\bibitem{PonteCastaneda:Suquet:1998}
P.~Ponte~Casta{\~n}eda and P.~Suquet.
\newblock Nonlinear composites.
\newblock {\em Advances in Applied Mechanics}, 34(998):171--302, 1998.

\bibitem{PonteCastaneda:Willis:1985}
P.~Ponte~Casta{\~n}eda and J.~R. Willis.
\newblock Variational second-order estimates for nonlinear composites.
\newblock {\em Proceedings of the Royal Society of London. Series A:
  Mathematical, Physical and Engineering Sciences}, 455(1985):1799--1811, 1999.

\bibitem{Potthast:2001}
R.~Potthast.
\newblock {\em Point sources and multipoles in inverse scattering theory},
  volume 427 of {\em Chapman \& Hall/CRC Research Notes in Mathematics}.
\newblock Chapman \& Hall/CRC, Boca Raton, FL, 2001.

\bibitem{Potthast:2005}
R.~Potthast.
\newblock Sampling and probe methods---an algorithmical view.
\newblock {\em Computing}, 75(2-3):215--235, 2005.

\bibitem{Ruzicka:2000}
M.~R\r{u}\v{z}i\v{c}ka.
\newblock {\em {Electrorheological fluids modeling and mathematical theory}}.
\newblock Number 1748 in Lecture Notes in Mathematics. Springer-Verlag, Berlin,
  2000.

\bibitem{Salo:Zhong:2012}
M.~Salo and X.~Zhong.
\newblock An inverse problem for the $p$-{L}aplacian: {B}oundary determination.
\newblock {\em {SIAM} J. Math. Anal.}, 44(4):2474--2495, Mar. 2012.

\bibitem{Armstrong:Henkelman:Joy:Scott:1991}
G.~Scott, M.~L.~G. Joy, R.~Armstrong, and R.~M. Henkelman.
\newblock Measurement of nonuniform current density by magnetic resonance.
\newblock {\em Medical Imaging, IEEE Transactions on}, 10(3):362--374, 1991.

\bibitem{Sini:Yoshida:2012}
M.~Sini and K.~Yoshida.
\newblock On the reconstruction of interfaces using complex geometrical optics
  solutions for the acoustic case.
\newblock {\em Inverse Problems}, 28(5):055013, 22, 2012.

\bibitem{Suquet:1993}
P.~Suquet.
\newblock Overall potentials and extremal surfaces of power law or ideally
  plastic composites.
\newblock {\em Journal of the Mechanics and Physics of Solids},
  41(6):981--1002, 1993.

\bibitem{Talbot:Willis:1994:a}
D.~R.~S. Talbot and J.~R. Willis.
\newblock Upper and lower bounds for the overall properties of a nonlinear
  composite dielectric. {I. R}andom microgeometry.
\newblock {\em Proceedings of the Royal Society of London. Series A:
  Mathematical and Physical Sciences}, 447(1930):365--384, 1994.
\newblock With second part \cite{Talbot:Willis:1994:b}.

\bibitem{Talbot:Willis:1994:b}
D.~R.~S. Talbot and J.~R. Willis.
\newblock Upper and lower bounds for the overall properties of a nonlinear
  composite dielectric. {II. P}eriodic microgeometry.
\newblock {\em Proceedings of the Royal Society of London. Series A:
  Mathematical and Physical Sciences}, 447(1930):385--396, 1994.
\newblock With first part \cite{Talbot:Willis:1994:a}.

\bibitem{Uhlmann:2009}
G.~Uhlmann.
\newblock Electrical impedance tomography and {C}alder{\'o}n's problem.
\newblock {\em Inverse problems}, 25(12):123011, 2009.

\bibitem{Wang:Zhou:2013}
J.-N. Wang and T.~Zhou.
\newblock Enclosure methods for {H}elmholtz-type equations.
\newblock In {\em Inverse problems and applications: inside out. {II}},
  volume~60 of {\em Math. Sci. Res. Inst. Publ.}, pages 249--270. Cambridge
  Univ. Press, Cambridge, 2013.

\bibitem{Wolff:2007}
T.~H. Wolff.
\newblock Gap series constructions for the {$p$-L}aplacian.
\newblock {\em Journal d'Analyse Mathematique}, 102(1):371--394, Aug. 2007.
\newblock Preprint written in 1984.

\bibitem{Zhou:2010}
T.~Zhou.
\newblock Reconstructing electromagnetic obstacles by the enclosure method.
\newblock {\em Inverse Probl. Imaging}, 4(3):547--569, 2010.

\end{thebibliography}
%\bibliography{mmA1}

\end{document}